\documentclass{amsart}
\usepackage{graphicx}
\usepackage{amsmath,amssymb}

\newtheorem{thm}{Theorem}[section]
\newtheorem{cor}[thm]{Corollary}

\theoremstyle{definition}

\theoremstyle{remark}
\newtheorem{rem}[thm]{Remark}
\theoremstyle{example}

\numberwithin{equation}{section}

\newcommand{\eps}{\varepsilon}

\newcommand{\del}{\delta}

\def\dlim{\displaystyle\lim}

\newcommand{\B}{{\mathbb B}}
\newcommand{\D}{{\mathbb D}}
\newcommand{\R}{{\mathbb R}}
\newcommand{\SSS}{{\mathbb S}}
\newcommand{\C}{{\mathbb C}}

\newcommand{\N}{{\mathbb N}}

\newcommand{\dsum}{\displaystyle\sum}
\newcommand{\dsup}{\displaystyle\sup}

\newcommand{\calN}{{\mathcal N}}

\newcommand{\calQ}{{\mathcal Q}}

\begin{document}

\title{$\calN_p$-type functions with Hadamard gaps in the unit ball}%

\date{\today}%

\author{Bingyang Hu and Songxiao Li}%

\address{ Bingyang  Hu: Division of Mathematical Sciences, School of Physical and Mathematical Sciences, Nanyang Technological University (NTU), 637371 Singapore}%
\email{bhu2@e.ntu.edu.sg}

\address{Songxiao Li:  Faculty of Information Technology,
Macau University of Science and Technology,  Avenida Wai Long,
Taipa, Macau. }%
\email{jyulsx@163.com }




\maketitle

\begin{abstract}
We study the holomorphic functions with Hadamard gaps in
$\calN_p$-spaces on the unit ball   of $\C^n$ when $0<p \le n$ and
$p>n$. A corollary on analytic functions with Hadamard gaps on
$\calN_p$-spaces on the unit disk is also given.
\end{abstract}


\section{Introduction}


Let $\B$ be the open unit ball in $\C^n$ with $\SSS$ as its boundary
and $H(\B)$ the collection of all holomorphic functions in $\B$.
$H^\infty$ denotes the Banach space consisting of all bounded
holomorphic functions in $\B$ with the norm $\|f\|_{\infty}=\sup_{z
\in \B} |f(z)|$. The  Bergman-type space $A^{-p}(\B)$ is the space
of all   $f\in H(\B)$ such that
$$|f|_p=\sup_{z \in \B} |f(z)|(1-|z|^2)^p<\infty.$$
Let $A^{-p}_0(\B)$ denote the closed subspace of  $A^{-p}(\B)$ such
that  $\lim_{|z|\rightarrow 1} |f(z)|(1-|z|^2)^p=0.$

The $\calN_p$-space in the unit disk $\D$ was first introduced in \cite{Pal} and studied in \cite{Uek}, which is defined as, for $p>0$,
$$
\calN_p(\D)=\left\{f \in H(\D): \|f\|_p=\sup_{a \in \D}
\left(\int_{\D}
|f(z)|^2(1-|\sigma_a(z)|^2)^pdA(z)\right)^{1/2}<\infty\right\},
$$
where $dA$ is the normalized area measure over $\D$ and
$\sigma_a(z)=\frac{a-z}{1-\bar{a}z} $ is a M\"obius transformation
of $\D$.

Let $dV$ denote the normalized volume measure over $\B$ and
$\Phi_a(z)$  the automorphism of $\B$ for $a \in \B$, i.e.,
$$ \Phi_a(z)=\frac{a-P_az-s_aQ_a z}{1-\langle z,a\rangle},
$$
where $s_a=  \sqrt{1-|a|^2} $, $P_a$ is the orthogonal projection
into the space spanned by $a$ and $Q_a=I-P_a$ (see, e.g.,
\cite{Zhu}).  The $\calN_p$-space on $\B$ was introduced in
\cite{HK1}, i.e.,
\[ \begin{split}
\calN_p&=\calN_p(\B)\\
&=\left\{f \in H(\B): \|f\|_p=\sup_{a \in \B} \left(\int_{\B}
|f(z)|^2(1-|\Phi_a(z)|^2)^pdV(z)\right)^{1/2}<\infty\right\}.\end{split}
\]   The little space of $\calN_p$-space, denoted by
$\calN_p^0$, which consisting of all $f \in \calN_p$ such that
$$
  \lim_{|a| \to 1^{-}} \int_{\B}
|f(z)|^2(1-|\Phi_a(z)|^2)^pdV(z) =0 .
$$

In \cite{HK1}, several basic properties of $\calN_p(\B)$-spaces are
proved, in connection with the Bergman-type spaces $A^{-q}$. In
particular, an embedding theorem for $\calN_p(\B)$-spaces and
$A^{-q}(\B)$ is obtained, together with other useful properties.

\begin{thm}\label{basic}\cite{HK1}
The following statements hold:
\begin{enumerate}
\item [(a)] For $p>q>0$, we have $H^{\infty}\hookrightarrow\calN_q\hookrightarrow\calN_p\hookrightarrow A^{-\frac{n+1}{2}}$.
\item [(b)] For $p>0$, if $p>2k-1, k\in(0,\frac{n+1}{2}]$, then $A^{-k}\hookrightarrow \calN_p$. In particular, when $p>n$, $\calN_p=A^{-\frac{n+1}{2}}$.
\item [(c)] $\calN_p$  is a functional Banach space with the norm $\|\cdot\|_p$, and moreover, its norm topology is stronger than the compact-open topology.
\end{enumerate}
\end{thm}

An   $f \in H(\B)$ written in the form
$$
f(z)=\sum_{k=0}^{\infty} P_{n_k}(z),
$$
where $P_{n_k}$ is a homogeneous polynomial of degree $n_k$, is said
to have \textit{Hadamard gaps} if for some $c>1$ (see. e.g.,
\cite{SS}),
$$
 \frac{n_{k+1}}{n_k}\ge c,\ \forall k \ge 0.
$$

Hadamard gaps series on spaces of holomorphic functions in $\D$ or
in $\B$ has been studied quite well. We refer the readers to the
related results in \cite{KLA1, KLA2, JSC, GP, LS, XM, JM, SS, SS1,
WZ, SY,  zz, Zhu2} and the reference therein.

The aim of the present paper is to characterize the holomorphic functions with Hadamarad gaps in $\calN_p$-space for
two different cases $0<p \le n$ and $p>n$. Our main results are contained in Section $2$.

Throughout this paper, for $a, b \in \R$, $a \lesssim b$ ($a \gtrsim
b$, respectively) means there exists a positive number $C$, which is
independent of $a$ and $b$, such that $a \leq Cb$ ($ a \geq Cb$,
respectively). Moreover, if both $a \lesssim b$ and $a \gtrsim b$
hold, then we say $a \simeq b$.


\section{Main results and proofs}



To formulate our main result, we denote
$$
M_k=\sup_{\xi \in \SSS} |P_{n_k}(\xi)| \quad \textrm{and} \quad
L_k=\left(\int_{\xi \in \SSS} |P_{n_k}(\xi)|^2d\sigma(\xi)
\right)^{1/2},
$$
where $d\sigma$ is the normalized surface measure on $\SSS$, that is, $\sigma(\SSS)=1$. Clearly for each $k \ge 0$, $M_k$ and $L_k$ are well-defined.


\subsection{The case when $0<p \le n$.}

In this subsection, we study the Hadamard gaps series in $\calN_p$-spaces when $0<p \le n$. We have the following result.

\begin{thm} \label{Hadamard01}
Let $0<p \le n$ and $f(z)=\dsum_{k=0}^{\infty} P_{n_k}(z)$ with
Hadamard gaps. Considering the following statements.
\begin{enumerate}
\item[(a)] $\dsum_{k=0}^{\infty} \bigg( \frac{1}{2^{k(1+p)}} \sum_{2^k \leq n_j <2^{k+1}} M_j^2\bigg)<\infty$;
\item[(b)] $f \in \calN_p^0$;
\item[(c)] $f \in \calN_p$;
\item[(d)] $\dsum_{k=0}^{\infty} \bigg( \frac{1}{2^{k(1+p)}} \sum_{2^k \leq n_j <2^{k+1}} L_j^2\bigg)<\infty$.
\end{enumerate}
We have $(a) \Longrightarrow (b) \Longrightarrow (c)
\Longrightarrow (d)$.
\end{thm}

\begin{proof}
 $\bullet \ (a) \Longrightarrow (b).$ Suppose that $(a)$ holds. First, we prove that $f \in \calN_p$.
  For $f(z)=\sum_{k=0}^{\infty} P_{n_k}(z)$, by using the polar coordinates and \cite[Lemma 1.8]{Zhu}, we have
\begin{eqnarray*}
&& \|f\|_p^2 =\sup_{a \in \B} \int_{\B} \left| \sum_{k=0}^{\infty} P_{n_k}(z)\right|^2(1-|\Phi_a(z)|^2)^pdV(z) \\
&\le&\sup_{a \in \B} \int_{\B} \left(\sum_{k=0}^{\infty}|P_{n_k}(z)|\right)^2 \frac{(1-|a|^2)^p(1-|z|^2)^p}{|1-\langle z, a \rangle|^{2p}}dV(z)  \\
&=& \sup_{a \in \B} \left\{(1-|a|^2)^p \int_{\B} \left(\sum_{k=0}^{\infty}|P_{n_k}(z)|\right)^2 \frac{(1-|z|^2)^p}{|1-\langle z, a \rangle|^{2p}}dV(z)  \right\}\\
&\le & 2n\sup_{a \in \B} \left\{(1-|a|^2)^p \int_0^1  \left(\sum_{k=0}^{\infty}|P_{n_k}(r\xi)|\right)^2(1-r^2)^p \left(\int_{\SSS} \frac{1}{|1-\langle r\xi, a\rangle|^{2p}}d\sigma(\xi)  \right)dr\right\}\\
&=& 2n \sup_{a \in \B} \left\{(1-|a|^2)^p \int_0^1  \left(\sum_{k=0}^{\infty}|P_{n_k}(\xi)|r^{n_k}\right)^2(1-r^2)^p \left(\int_{\SSS} \frac{1}{|1-\langle r\xi, a\rangle|^{2p}}d\sigma(\xi)  \right)dr\right\}\\
&\le& 2n \sup_{a \in \B} \left\{(1-|a|^2)^p \int_0^1
\left(\sum_{k=0}^{\infty}M_kr^{n_k}\right)^2(1-r^2)^p
\left(\int_{\SSS} \frac{1}{|1-\langle r\xi,
a\rangle|^{2p}}d\sigma(\xi) \right)dr\right\}.
\end{eqnarray*}

Applying \cite[Theorem 1.12]{Zhu}, for each $a \in \B$ and $r \in [0, 1]$, we have
\begin{eqnarray*}
\int_{\SSS} \frac{1}{|1-\langle r\xi, a\rangle|^{2p}}d\sigma(\xi)%
&=& \int_{\SSS} \frac{1}{|1-\langle \xi, ar\rangle|^{2p}}d\sigma(\xi)\\
&=& \int_{\SSS} \frac{d\sigma(\xi)}{|1-\langle ar, \xi\rangle|^{n+(2p-n)}}\\
&\simeq& \begin{cases}
\textrm{bounded in} \ \B, &0<p<\frac{n}{2},\\
\log\frac{1}{1-r^2|a|^2} \le \log\frac{1}{1-|a|^2}, & p=\frac{n}{2},\\
(1-r^2|a|^2)^{n-2p} \le (1-|a|^2)^{n-2p}, &\frac{n}{2}<p<n.
\end{cases}
\end{eqnarray*}
It is clear that, for all cases of $p$, we can have
$$
(1-|a|^2)^p\int_{\SSS} \frac{1}{|1-\langle r\xi, a\rangle|^{2p}}d\sigma(\xi) \le M, \quad a \in \B,
$$
where $M$ is a positive number independent of both $a$ and $r$. Now,
applying \cite[Theorem 1]{MM}, we have
\begin{eqnarray*}
\|f\|_p^2%
&\leq& 2nM  \int_0^1  \left(\sum_{k=0}^{\infty}M_kr^{n_k}\right)^2(1-r^2)^p dr\\
&\simeq& 2nM \sum_{k=0}^{\infty} \frac{1}{2^{k(1+p)}} \left( \sum_{2^k \leq n_j<2^{k+1}} M_j\right)^2.
\end{eqnarray*}
Since $f$ is in the Hadamard gaps class, there exists a constant $c>1$ such that $n_{j+1} \geq cn_j$ for all $j \ge 0$. Hence,
 the maximum number of $n_j$'s between $2^k$ and $2^{k+1}$   is less or equal to $[\log_c 2]+1$ for  $k=0,1,2,\cdot\cdot\cdot$.

Since for every $ k \ge 0$, by Cauchy-Schwarz inequality,
$$
\bigg(\sum_{2^k \leq n_j<2^{k+1}} M_j \bigg)^2 \leq   \Big([\log_c
2]+1\Big)\bigg(\sum_{2^k \leq n_j<2^{k+1}} M_j^2\bigg).
$$
  Thus, we have
\begin{equation} \label{Theorem205}
\|f\|_p^2 \lesssim
\sum_{k=0}^{\infty}\bigg(\frac{1}{2^{k(1+p)}}\sum_{2^k\leq n_j
<2^{k+1}} M_j^2\bigg)<\infty,
\end{equation}
which implies $f \in \calN_p$.

Next, we prove that $f \in \calN_p^0$. Let
$$
A= \int_{\B}(1-|\Phi_a(z)|^2)^pdV(z).
$$
We claim that $A \to 0$ as $|a| \to 1^{-}$. By 
\cite[Lemma 1.2]{Zhu}, we have
\begin{eqnarray*}
 \int_{\B}(1-|\Phi_a(z)|^2)^pdV(z)%
&=& \int_{\B} \frac{(1-|a|^2)^p(1-|z|^2)^p}{|1-\langle z, a\rangle |^{2p}}dV(z)\\
&=&(1-|a|^2)^p \int_{\B} \frac{(1-|z|^2)^p}{|1-\langle z,
a\rangle|^{n+1+p+p-n-1}}dV(z).
\end{eqnarray*}
Applying \cite[Theorem 1.12]{Zhu}, we know
\[ \int_{\B} \frac{(1-|z|^2)^{p}}{|1-\langle a, z\rangle|^{n+1+p+p-n-1}}dV(z) \simeq \begin{cases}
\textrm{bounded in} \ \B, & 0<p<n+1,\\
\log\frac{1}{1-|a|^2}, & p=n+1,\\
(1-|a|^2)^{1+n-p}, &p>n+1.
\end{cases} \]
It is clear no matter for what case, $A \to 0$~ as~ $ a \to 1^{-}$.

Put $f_m(z)=\sum_{k=0}^m P_{n_k}(z), m \in \N$ and $K_m=\max\{M_0, M_1, \dots, M_m\}$. Note that for each $a \in \B$,
\begin{eqnarray*}
&& \int_\B |f_m(z)|^2(1-|\Phi_a(z)|^2)^pdV(z)\\
&\le& \int_\B \bigg( \sum_{k=0}^m|P_{n_k}(z)|\bigg)^2(1-|\Phi_a(z)|^2)^pdV(z)\\
&\le& m^2K_m^2 \int_\B (1-|\Phi_a(z)|^2)^pdV(z),
\end{eqnarray*}
which tends to $0$ as $|a| \to 1^{-}$. Hence, $f_m \in \calN_p^0$. Moreover,  by \cite[Corollary 2.6]{HKL}, $\calN_p^0$ is closed and the set of all polynomials is dense in $\calN_p^0$, and hence it suffices to show that $\|f_m-f\|_p \to 0$ as $m \to \infty$. By \eqref{Theorem205}, we have
\begin{equation} \label{Theorem2051}
\|f_m-f\|_p^2 \lesssim
\sum_{k=m'}^{\infty}\bigg(\frac{1}{2^{k(1+p)}}\sum_{2^k\leq n_j
<2^{k+1}} M_j^2\bigg),
\end{equation}
where $m'=\left[ \frac{m+1}{[\log_c 2]+1} \right]$. The result follows from condition (a) and \eqref{Theorem2051}.

\medskip

$\bullet \  (b) \Longrightarrow (c).$ It is obvious.

\medskip

$\bullet \ (c) \Longrightarrow (d).$ Suppose $f \in \calN_p$. As the proof in \cite[Theorem 1]{SS}, we have
\begin{eqnarray*}
\|f\|_p^2%
&=&\sup_{a \in \B} \int_{\B} \left| \sum_{k=0}^{\infty} P_{n_k}(z) \right|^2 (1-|\Phi_a(z)|^2)^pdV(z)\\
&\ge&  \int_{\B} \left| \sum_{k=0}^{\infty} P_{n_k}(z) \right|^2 (1-|z|^2)^pdV(z)\\
&\simeq& \int_{\SSS} \bigg (\sum_{k=0}^\infty \frac{1}{2^{k(1+p)}} \sum_{2^k \leq n_j<2^{k+1}} |P_{n_k}(\xi)|^2 \bigg) d\sigma(\xi)\\
&=& \sum_{k=0}^{\infty} \bigg(\frac{1}{2^{k(1+p)}} \sum_{2^k \leq
n_j <2^{k+1}} L_j^2 \bigg),
\end{eqnarray*}
which implies the desired result.
\end{proof}

\begin{rem}
Generally, when $n>1$, the above conditions in Theorem \ref{Hadamard01} are not equivalent. For example, $(d) \nRightarrow (a)$. Indeed, put
$$
f(z)=\sum_{k=0}^\infty 2^{\frac{k(p+1)}{2}}z_1^{2^k}, \quad z=(z_1, z_2, \dots, z_n) \in \B.
$$
Since $M_k=2^{\frac{k(p+1)}{2}}$, we have
$$
\sum_{k=0}^\infty \bigg (\frac{1}{2^{k(1+p)}} \sum_{2^k \leq n_j
<2^{k+1}} M_k^2\bigg)=\infty.
$$
On the other hand, by \cite[Lemma 1.11]{Zhu}, for each $k \ge 0$, we
have
$$
L_k^2=2^{k(p+1)} \int_{z \in \SSS} |z_1^{2^k}|^2d\sigma(z) = 2^{k(p+1)} \cdot \frac{(n-1)! (2^k)!}{(n-1+2^k)!} \lesssim \frac{2^{k(p+1)}}{2^{k(n-1)}},
$$
which implies
$$
\sum_{k=0}^\infty \bigg (\frac{1}{2^{k(1+p)}} \sum_{2^k \leq n_j
<2^{k+1}} L_k^2\bigg) \lesssim \sum_{k=0}^\infty
\frac{1}{2^{k(n-1)}}<\infty.
$$
\end{rem} 

Next, we consider some special cases when all the conditions in Theorem \ref{Hadamard01} are equivalent.

In \cite{RW}, the authors constructed a sequence of homogeneous
polynomial $\{T_k\}_{k \in \N}$ satisfying $\deg(T_k)=k$,
\begin{equation} \label{RWeq}
\sup_{\xi \in \SSS} |T_k(\xi)|=1 \quad \textrm{and} \quad \int_{\xi
\in \SSS} |T_k(\xi)|^2d\sigma(\xi) \ge \frac{\pi}{2^{2n}}.
\end{equation}
An immediate corollary of Theorem \ref{Hadamard01} is stated as follows.

\begin{cor}
Let $0<p \le n$ and $f(z)=\sum_{k=0}^\infty a_kT_{n_k}(z)$ with
Hadamard gaps, where $a_k \in \C, k \ge 0$. Then the following statements are equivalent.
\begin{enumerate}
\item[(a)] $\dsum_{k=0}^{\infty} \left( \frac{1}{2^{k(1+p)}} \sum_{2^k \leq n_j <2^{k+1}} |a_j|^2\right)<\infty$;
\item[(b)] $f \in \calN_p^0$;
\item[(c)] $f \in \calN_p$.
\end{enumerate}
\end{cor}

\begin{proof}
The desired result follows from the fact that for each $k \ge 0$, $
M_k \simeq L_k.$
\end{proof}
Moreover, letting $n=1$, we have the following corollary decribing the functions in $\calN_p(\D)$ with Hadamard gaps.

\begin{cor} \label{cor01}
Let $0<p \le 1$ and $f(z)=\dsum_{k=0}^\infty b_{k}z^{n_k}$ with
Hadamard gaps,  where $b_k \in \C, k \ge 0$. The the following conditions are equivalent.
\begin{enumerate}
\item[(a)] $f \in \calN_p(\D)$;
\item[(b)] $f \in \calN_p^0(\D)$;
\item[(c)] $\dsum_{k=0}^\infty \bigg(\frac{1}{2^{k(1+p)}} \dsum_{2^k \le n_j<2^{k+1}} |b_j|^2\bigg)<\infty$.
\end{enumerate}
Note that the result in \cite[Theorem 3.3 (a)]{Pal} is contained in Corollary \ref{cor01}.
\end{cor}

\begin{proof}
The desired result follows from the fact that when $n=1$,
$M_j=L_j=|b_{j}|$.
\end{proof}

\vskip 3mm
\subsection{The case when $p>n$.}

By Theorem \ref{basic}, when $p>n$, all $\calN_p$-spaces coincide
with $A^{-\frac{n+1}{2}}$. In this subsection, we consider   a more
general question about the   Hadamard gaps series in $A^{-l}$ for
any $l>0$. We have the following result.

\begin{thm} \label{HadBtype}
Let $l>0$ and $f(z)=\sum_{k=0}^\infty P_{n_k}(z)$ with Hadamard
gaps, where $P_{n_k}$ is a homogeneous polynomial of degree $n_k$.
Then the following assertions hold.
\begin{enumerate}
\item[(a)] $f \in A^{-l}$ if and only if $\dsup_{k \geq 1} \frac{M_k}{n_k^l}<\infty$;
\item[(b)] $f \in A^{-l}_0$ if and only if $\dlim_{k \to \infty} \frac{M_k}{n_k^l}=0$.
\end{enumerate}
Note that the result in \cite[Theorem 3.3 (b)]{Pal} is a particular case of the assertion $(a)$ in Theorem \ref{HadBtype}.
\end{thm}

\begin{proof}

(a) \textbf{Necessity.} Suppose $f \in A^{-l}$. Fix a $\xi \in \SSS$ and denote
$$
f_\xi(w)=\sum_{k=0}^\infty P_{n_k}(\xi) w^{n_k}=\sum_{k=0}^\infty
P_{n_k}(\xi w), ~~ w \in \D.
$$
Since $f \in H(\B)$, it known that for a fixed $\xi \in
\SSS$, $f_\xi(w)$ is holomorphic in $\D$ (see, e.g., \cite{Rud}).
Hence, for any $r \in (0,
1)$, we have
\begin{eqnarray} \label{ineq0500}
M_k%
&=& \sup_{ \xi \in   \SSS} |P_{n_k}(\xi)|=\sup_{ \xi \in   \SSS} \left |\frac{1}{2\pi i} \int_{|w|=r} \frac{f_\xi(w)}{w^{n_k+1}}dw\right| \\
&=& \frac{1}{2\pi} \sup_{ \xi \in \SSS}\left |\int_{|w|=r} \frac{f( \xi w)}{w^{n_k+1}}dw\right|  \nonumber \\
&\le& \frac{1}{2\pi}\sup_{ \xi \in   \SSS} \int_{|w|=r} \frac{|f( \xi w)|}{r^{n_k+1}}|dw|  \nonumber \\
& \le & \frac{1}{2\pi}\sup_{ \xi \in  \SSS} \int_{|w|=r} \frac{|f( \xi w)| (1-|\xi w|^2)^{l}}{r^{n_k+1}(1-r)^{l}}|dw| \nonumber \\
&\le& \frac{|f|_l}{r^{n_k}(1-r)^{l}}. \nonumber
\end{eqnarray}

In \eqref{ineq0500}, letting $r=1-\frac{1}{n_k}$, we have
$$
M_k \le \frac{|f|_l \cdot n_k^l}{(1-\frac{1}{n_k})^{n_k}}.
$$
Thus, for each $k \ge 2$,
$$
\frac{M_k}{n_k^l} \le \frac{|f|_l}{(1-\frac{1}{n_k})^{n_k}} \le 4|f|_l,
$$
which implies that
$$
\sup_{k \ge 1} \frac{M_k}{n_k^l} \le \max \left\{ \frac{M_1}{n_1^l},  4|f|_l \right\}<\infty.
$$

\medskip

\textbf{Sufficiency.} Suppose that $\sup_{k \ge 1} \frac{M_k}{n_k^l}<\infty$. Then
$$
|f(z)|=\left|\sum_{k=0}^\infty P_{n_k} \left(\frac{z}{|z|}\right) |z|^{n_k}\right| \leq \sum_{k=0}^{\infty} M_k|z|^{n_k} \lesssim \sum_{k=0}^{\infty} n_k^l |z|^{n_k}.
$$
Thus,
$$
\frac{|f(z)|}{1-|z|} \lesssim \bigg(\sum_{k=0}^{\infty} n_k^l
|z|^{n_k}\bigg) \bigg(\sum_{s=0}^{\infty} |z|^s
\bigg)=\sum_{t=0}^{\infty} \bigg(\sum_{n_j \leq t}n_j^l \bigg)
|z|^t.
$$
Since
$$
\lim_{k \to \infty} \frac{k^l k!}{l (l+1) \dots (l+k)}=\Gamma (l),  \ l>0,
$$
we have
$$
\sup_{k \in \N} \left(\frac{k^l k!}{(k+l)(k+l-1) \dots (l+1)}\right) \leq M,
$$
where $M$ is a positive number depending on $l$. Hence, we have for
each $k \ge 0$,
\begin{eqnarray} \label{ineq009}
\frac{k^l}{(-1)^k {-l-1 \choose k}}%
&=&\frac{k^l k!}{(-1)^k (-l-1)(-l-2) \dots (-l-k)} \nonumber\\
&=&\frac{k^l k!}{(k+l)(k+l-1) \dots (l+1)} \leq M,
\end{eqnarray}
where ${\alpha \choose k}=\frac{\alpha(\alpha-1)\dots (\alpha-k+1)}{k!}, \alpha \in \R$.

Moreover, since $f$ is in Hadamard gaps class, there exists a constant $c>1$ such that $n_{j+1} \geq cn_j$ for all $j \ge 0$. Hence
\begin{equation} \label{ineq0091}
 \frac{1}{k^l} \bigg(\sum_{n_j \leq k} n_j^l \bigg) \leq \sum_{m=0}^\infty \bigg(\frac{1}{c^l}\bigg)^m =\frac{c^l}{c^l-1}.
\end{equation}
Combining \eqref{ineq009} and \eqref{ineq0091}, we have
$$
\frac{k^l}{(-1)^k {-l-1 \choose k}} \cdot \frac{1}{k^l}
\bigg(\sum_{n_j \leq k} n_j^l\bigg)  \leq \frac{Mc^l}{c^l-1},
$$
which implies
\begin{equation} \label{ineq010}
\sum_{n_j \leq k} n_j^l \le (-1)^k {-l-1 \choose k}
\frac{Mc^l}{c^l-1}.
\end{equation}

Hence, for any $z\in \B$, by \eqref{ineq010} we have
$$
\frac{|f(z)|}{1-|z|} \lesssim \frac{Mc^l}{c^l-1} \cdot
\sum_{t=0}^{\infty} (-1)^t  {-l-1 \choose t}
|z|^t=\frac{Mc^l}{c^l-1} \cdot \frac{1}{(1-|z|)^{l+1}},
$$
which implies
$$
|f(z)|(1-|z|^2)^l \lesssim \frac{Mc^l}{c^l-1},
$$
and hence $f \in A^{-l}$.

\medskip

(b) \textbf{Necessity.} Suppose $f \in A^{-l}_0$, that is, for any $\varepsilon>0$, there exists a $\del \in (0, 1)$, when $\del<|z|<1$,
$$
|f(z)|(1-|z|^2)^l< \varepsilon.
$$
Take $N_0 \in \N$ satisfying $\del<1-\frac{1}{n_k}<1$ when $k>N_0$.
Then for any $k > N_0$ and $r=1-\frac{1}{n_k}$, applying the proof
in part $(a)$, we have
$$
M_k \le \frac{n_k^l}{ (1-\frac{1}{n_k})^{n_k}}  \cdot \sup_{
\del<|z|<1}  |f(z)|(1-|z|^2)^l   < \frac{ \varepsilon n_k^l}{
(1-\frac{1}{n_k})^{n_k}},
$$
which implies
$$
\frac{M_k}{n_k^l} \le \frac{\varepsilon}{ (1-\frac{1}{n_k})^{n_k}}
\le 4 \varepsilon, ~~k>N_0.
$$
 Hence we have  $\lim_{k \to \infty} \frac{M_k}{n_k^l}=0$.

\textbf{Sufficiency.} Since $\lim_{k \to \infty}
\frac{M_k}{n_k^l}=0$, it is clear that $\sup_{k \ge 1}
\frac{M_k}{n_k^l}<\infty$ and hence by part (a), we have $f \in
A^{-l}$. For any $\varepsilon>0$, there exists a $N_0 \in \N$
satisfying
$$
\frac{M_m}{n_m^l}<\varepsilon,
$$
when $m>N_0$.  For each $m \in \N$, put $f_m(z)=\sum_{k=0}^m
P_{n_k}(z)$. Note that
\begin{eqnarray*}
|f_m(z)|(1-|z|^2)^l%
&\le& \bigg(\sum_{k=0}^m |P_{n_k}(z)|\bigg)(1-|z|^2)^l \\
&=&\left(\sum_{k=0}^m \left| P_{n_k}\left(\frac{z}{|z|}\right) |z|^{n_k}\right|\right)(1-|z|^2)^l \\
&\le& K(1-|z|^2)^l \sum_{k=0}^m |z|^{n_k} \le Km(1-|z|^2)^l,
\end{eqnarray*}
where $K=\max\{M_0, M_1, M_2, \dots, M_m\}$. Hence, $\lim_{|z| \to 1^{-}} |f_m(z)|(1-|z|^2)^l=0$, that is, for each
$m \in \N$, $f_m \in A^{-l}_0$ and hence it suffices to show that $|f_m-f|_l \to 0$ as $m \to \infty$. Indeed, for $m>N_0$, we have
$$
|f_m(z)-f(z)|= \left| \sum_{k=m+1}^\infty P_{n_k}(z)\right| \le \sum_{k=m+1}^\infty M_k |z|^{n_k}\le \varepsilon \sum_{k=m+1}^\infty n_k^l |z|^{n_k}.
$$
Applying the proof in part (a), we have
\begin{eqnarray*}
\frac{|f_m(z)-f(z)|}{1-|z|}%
&\le& \eps \bigg(\sum_{k=m+1}^\infty n_k^l |z|^{n_k} \bigg)  \bigg( \sum_{s=0}^\infty |z|^s\bigg)=
\eps \sum_{l=n_{m+1}}^\infty \bigg(\sum_{n_{m+1} \le n_j \le l} n_j^l \bigg) |z|^l\\
&\le& \eps \sum_{t=0}^\infty \bigg(\sum_{n_j \le t} n_j^l \bigg)
|z|^t \le M'\frac{\eps}{(1-|z|)^{l+1}},
\end{eqnarray*}
where $M'$ is a positive number independent of $m$. Hence, when
$m>N_0$, we have $|f_m-f|_l \le M' \varepsilon$, which implies that
$f \in A^{-l}_0$.
\end{proof}

\vskip 3mm

 \noindent {\bf Acknowledgments.}  The second author was
supported by the Macao Science and Technology Development
Fund (No.083/2014/A2).

\end{document}